\documentclass[11pt,oneside,leqno]{amsart}

\usepackage{amsxtra}
\usepackage{amsopn}
\usepackage{amsmath,amsthm,amssymb}
\usepackage{color}
\usepackage{amscd}
\usepackage{pifont}
\usepackage{amsfonts}
\usepackage{latexsym}
\usepackage{verbatim}
\usepackage{pb-diagram}

\theoremstyle{plain}
\newtheorem{theorem}{Theorem}[section]
\newtheorem*{theorem*}{Theorem}

\newtheorem{prop}[theorem]{Proposition}
\newtheorem{cor}[theorem]{Corollary}
\newtheorem{rem}[theorem]{Remark}

\newtheorem*{mt*}{Main Theorem}
\newcommand\R{{\mathbb R}}

\newcommand\pa[1]{\partial_{#1}}

\begin{document}
\title[Totally Geodesic and Parallel Hypersurfaces of Siklos Spacetimes]{Totally Geodesic and Parallel Hypersurfaces of Siklos Spacetimes}
\author{Giovanni Calvaruso}
\address{Giovanni Calvaruso: Dipartimento di Matematica e Fisica \lq\lq E. De Giorgi\rq\rq \\
Universit\`a del Salento\\
Prov. Lecce-Arnesano \\
73100 Lecce\\ Italy.}
\email{giovanni.calvaruso@unisalento.it}
\author{Lorenzo Pellegrino}
\address{Lorenzo Pellegrino: Dipartimento di Matematica e Fisica \lq\lq E. De Giorgi\rq\rq \\
Universit\`a del Salento\\
Prov. Lecce-Arnesano \\
73100 Lecce\\ Italy.}
\email{lorenzo.pellegrino@unisalento.it}
\author{Joeri Van der Veken}
\address{Joeri Van der Veken: Department of Mathematics \\
University of Leuven\\
Celestijnenlaan 200B \\
3001 Leuven \\  Belgium.}
\email{joeri.vanderveken@kuleuven.be}

\thanks{G. Calvaruso is partially supported by funds of GNSAGA (INDAM). J. Van der Veken is supported by the Research Foundation–Flanders (FWO) and the National Natural Science Foundation of China (NSFC) under collaboration project G0F2319N, by the KU Leuven Research Fund under project 3E210539 and by the Research Foundation–Flanders (FWO) and the Fonds de la Recherche Scientifique (FNRS) under EOS Projects G0H4518N and G0I2222N}

\begin{abstract}
We {classify} and describe totally geodesic and parallel hypersurfaces for the entire class of Siklos spacetimes. A large class of minimal hypersurfaces is also described.
 \end{abstract}

\maketitle

\section{Introduction}

Siklos spacetimes \cite{Si} are a well-known class of solutions to Einstein's field equations with Einstein-Maxwell source. They fall within type $N$ in { Petrov's classification}, have a negative cosmological constant and admit a  null non-twisting Killing field.  An arbitrary Siklos metric is given by
\begin{equation}\label{Smet}
g=  \frac{\beta^2}{x_3 ^2}\left(2dx_1 dx_2 +H dx_2 ^2+dx_3 ^2+dx_4 ^2 \right)
\end{equation}
{with} respect to some global coordinates $(x_1,x_2,x_3,x_4)$, where $H=H(x_2,x_3,x_4)$ is any smooth function of three variables (see \cite{Si}, \cite{Pod}) and $\beta^2=-\frac{3}{\Lambda}$, with $\Lambda <0$ denoting the cosmological constant.

Since their introduction, several relevant geometric and physical properties of Siklos metrics have been studied. In particular, Siklos spacetimes have been interpreted as exact gravitational waves propagating in the anti-de Sitter universe \cite{Pod}. Moreover, they correspond to the subclass $(IV)_0$ of  Kundt spacetimes {\cite{ORR}}. Further results on the geometry of Siklos spacetimes may be found in \cite{ChomS}--\cite{CnoteMS}, \cite{CZ}  and references therein. Finally, some well-known examples of homogeneous spacetimes introduced and studied by Defrise \cite{Defr},  Kaigorodov 
\cite{Kai},  Ozsv\'{a}th \cite{Ozs} fall within the class of Siklos metrics (see also the last section of this paper).

A submanifold $M$ of a pseudo-Riemannian manifold $\bar M$ is said to be \textit{totally geodesic} when its second fundamental form $h$ vanishes identically. As a consequence, a totally geodesic submanifold is characterized by the geometric property that its geodesics are also geodesics in the ambient space. More in general, a submanifold $M$ is said to be \textit{parallel} when its second fundamental form $h$ is covariantly constant. Consequently, all the extrinsic invariants of $M$ are covariantly constant. The classification of parallel and totally geodesic hypersurfaces of a given pseudo-Riemannian manifold is a natural problem, which enriches our understanding of the properties of the ambient space.  

Classifications of parallel surfaces have been obtained in several three-dimensional Lorentzian ambient spaces (see for example \cite{CV1}--\cite{CV4}). Also because of their relevance in Theoretical Physics, four-dimensional Lorentzian manifolds are natural candidates for this kind of study, but in a four-dimensional ambient space the investigation of parallel and totally geodesic hypersurfaces {is clearly more difficult than the corresponding three-dimensional problem}. Parallel hypersurfaces {of} some four-dimensional Lorentzian and pseudo-Riemannian manifolds were classified in \cite{CPV}, \cite{CSV}, \cite{CV5}, \cite{DV}.  

{Minimal hypersurfaces are a well-known and natural generalization of totally geodesic hypersurfaces. Instead of the vanishing of the second fundamental form, minimal hypersurfaces are defined by the the vanishing of the trace of $h$. {A further generalization is given by {\em hypersurfaces of constant mean curvature (CMC)}, for which the trace of $h$ is constant. Minimal and CMC hypersurfaces are intensively studied in different contexts}. Minimal hypersurfaces carry a natural and  relevant geometric meaning, as they are stationary with respect to the variation of the induced volume 
\cite{Seno}. Moreover, timelike minimal Lorentzian hypersurfaces appear in some natural ways in in Theoretical Physics. In fact,
they can be seen as simple but nontrivial examples of $D$-branes \cite{Allen} and in the Minkowski spacetime they {are} \lq\lq kinks\rq\rq (timelike submanifolds containing concentrations of Lagrangian density) \cite{Neu}.  }

The aim of this paper is to achieve a complete classification and description of parallel and totally geodesic hypersurfaces for general Siklos metrics. As a byproduct, we shall also explicitly describe {some classes of minimal and CMC hypersurfaces for  large families of Siklos spacetimes.} 
The paper is organized in the following way. In Section~2 we give the Levi-Civita connection and curvature of Siklos metrics and provide some basic information on parallel and totally geodesic hypersurfaces.  
In Section~3 we study hypersurfaces of Siklos spacetimes admitting a Codazzi second fundamental form, which include parallel (and { hence,} totally geodesic) hypersurfaces. { We also show that all the examples within a large class of hypersurfaces with a Codazzi second fundamental form are minimal hypersurfaces. We then give the complete classification of totally geodesic hypersurfaces of Siklos spacetimes. Finally, in Section~4 we {investigate parallel hypersurfaces of Siklos spacetimes. {We describe a large class of proper parallel hypersurfaces and also show that they are CMC. Moreover, we achieve a complete classification of proper parallel hypersurfaces for a subclass of Siklos spacetimes containing several well-known homogeneous examples}.

\section{Preliminaries}
\setcounter{equation}{0}

\subsection{On the curvature and connection of Siklos metrics}
Let $g$ be the metric described by  \eqref{Smet} with respect to the coordinate system $(x_1,x_2,x_3,x_4)${ ($=(v,u,x,y)$ in the original paper \cite{Si})}.
We shall denote by $\{ \pa {_i}=\frac{\partial}{\partial x_i}\}$ the basis of coordinate vector fields. {  By 
\eqref{Smet}, Siklos metrics $g$ are completely determined by 
$$
g(\pa1,\pa2)=g(\pa3,\pa3)=g(\pa4,\pa4)=\frac{\beta^2}{x_3^2}, \qquad g(\pa2, \pa2)=\frac{\beta^2}{x_3^2}H.
$$
In addition denote} by $H'_i =\partial _{i} H, H''_{ij}=\partial^2 _{ij} H,\ldots$ the partial derivatives of the defining function $H$.

By means of the {\em Koszul formula} (see also \cite{ChomS}), the Levi-Civita connection $\nabla$ of $g$ is explicitly described by the following possibly non-vanishing components
\begin{equation}\label{nabla}
\begin{array}{ll}
\nabla_{\pa1}\pa2=\frac{1}{x_3}\pa 3, &\;
\nabla_{\pa1}{ \pa3}=-\frac{1}{x_3}\pa1,
 \\[6pt]
\nabla_{\pa2}\pa2=\frac{1}{2}H'_2\pa1+\frac{1}{2x_3}(2H-x_3H'_3)\pa3-\frac{1}{2}H'_4\pa4, &\;
\nabla_{\pa2}\pa4=\frac{1}{2}H'_4\pa1, 
 \\[6pt]
  \nabla_{\pa2}\pa3=\frac{1}{2}H'_3\pa1-\frac{1}{x_3}\pa2, &\;
\nabla_{\pa3}\pa4=-\frac{1}{x_3}\pa4,
 \\[6pt]
  \nabla_{\pa3}\pa3=-\frac{1}{x_3}\pa3, &\;
\nabla_{\pa4}\pa4=\frac{1}{x_3}\pa3.
\end{array}
\end{equation}

{Starting} from \eqref{nabla}, a {direct} calculation yields that with respect to the basis $\{\pa i \}$ of coordinate vector fields, {(taking into account its symmetries) the curvature tensor  is completely determined  by the following possibly 
non-vanishing components}
\begin{equation}\label{R}
\begin{array}{ll}
R(\pa1,\pa2)\pa1= -\frac{1}{x_3^2} \pa1, &\; 
R(\pa1,\pa2)\pa2=-\frac{H}{x_3^2} \pa1+\frac{1}{x_3^2}  \pa2,
\\[7pt]
R(\pa1,\pa3)\pa2=\frac{1}{x_3^2}  \pa3, &\; 
R(\pa1,\pa3)\pa3=-\frac{1}{x_3^2} \pa1,
\\[7pt]
R(\pa1,\pa4)\pa2=\frac{1}{x_3^2}  \pa4, &\;
R(\pa1,\pa4)\pa4=-\frac{1}{x_3^2} \pa1,
\\[7pt]
R(\pa2,\pa3)\pa1=\frac{1}{x_3^2} \pa3, &\;
R(\pa2,\pa3)\pa2=\frac{f_1}{2x_3^2}\pa 3+\frac{1}{2} H''_{34} \pa 4,
\\[7pt]
R(\pa2,\pa3)\pa3=-\frac{f_2}{2x_3^2}\pa 1-\frac{1}{x_3^2}\pa 2, &\;
R(\pa2,\pa3)\pa4=-\frac{1}{2}H''_{34}\pa 1,
\\[7pt]
R(\pa2,\pa4)\pa1=\frac{1}{x_3^2} \pa4, &\;
R(\pa2,\pa4)\pa2=\frac{f_3}{2x_3^2}\pa 4+\frac{1}{2}H''_{34}\pa 3,
\\[7pt]
R(\pa2,\pa4)\pa3=-\frac{1}{2}H''_{34}\pa 1, &\;
R(\pa2,\pa4)\pa4=-\frac{f_4}{2x_3^2}\pa 1-\frac{1}{x_3^2}\pa 2, 
\\[7pt]
R(\pa3,\pa4)\pa3=\frac{1}{x_3^2}  \pa4, &\;
R(\pa3,\pa4)\pa4=-\frac{1}{x_3^2} \pa3,
\end{array}
\end{equation}
where {for simplicity} we put
{ \begin{equation}\label{fi}
\begin{array}{ll}
f_1= 2H-x_3 H'_3+x^2_3 H''_{33}, &\quad f_2=x_3 H''_{33}-H'_3, 
\\[7pt]
f_3= 2H-x_3 H'_3+x^2_3 H''_{44}, &\quad f_4=x_3 H''_{44}-H'_3.
\end{array}
\end{equation}}
We must point out that, from now and throughout this paper, we will always use the sign convention $R(X,Y)=[\nabla_X, \nabla_Y]-\nabla_{[X,Y]}$.

We may refer to either \cite{Pod} or \cite{ChomS} for further details on the geometry {of} Siklos spacetimes. Here we only recall that a Siklos metric $g$, defined by an arbitrary function $H=H(x_2,x_3,x_4)$,
\begin{itemize}
\item is Einstein if and only if  $\frac{2}{x_3} H'_3 -  H''_{33} - H''_{44} =0$;
\item is (locally) conformally flat if and only if $H''_{33} - H''_{44} =H''_{34} =0$;
\item has constant sectional curvature if and only if both the above sets of equations are satisfied. In this case, $(\bar M,g)$ is isometric to the anti-de Sitter space.
\end{itemize}

\subsection{On parallel and totally geodesic hypersurfaces}

Let $F : M^n \to \bar M^{n+1}$ be an isometric immersion of pseudo-Riemannian manifolds. Consider a unit normal vector field  $\xi$ along the image of $F$ {with} $g(\xi,\xi) = \varepsilon \in \{ -1,1\}$. Let $\nabla ^M$ and $\nabla$ denote the Levi-Civita connections of $M^n$ and $\bar M^{n+1}$ respectively. The well known {\em formula of Gauss} {states}
\begin{align} \label{fG}
\nabla_X Y = \nabla ^M _X Y + h(X,Y) \xi ,
\end{align}
for all vector fields $X, Y$ tangent to $M^n$ (we will always identify vector fields
tangent to {{ $M^n$} with their images} under $dF$). The formula of Gauss defines the \emph{second fundamental form} $h$ of the immersion, which is a symmetric $(0,2)$-tensor field on $M^n$.

{ $M^n$} is a \emph{totally geodesic hypersurface} if $h=0$. This is equivalent to requiring that every geodesic of { $M^n$} is also a geodesic of the ambient space { $\bar M^{n+1}$}. 

The covariant derivative $\nabla^M h$ is given by
$$
(\nabla^M h)(X,Y,Z) = X(h(Y,Z)) - h(\nabla^M _XY,Z) - h(Y,\nabla ^M _XZ),
$$
for all vector fields $X,Y,Z$ tangent to { $M^n$}. The hypersurface is said to be \emph{parallel} {(or to have parallel
second fundamental form)} if 
\begin{align} \label{nablahis0}
\nabla^M h = 0.
\end{align}
Clearly, totally geodesic hypersurfaces are parallel.



Let $R^M$ and $R$ denote the Riemann-Christoffel curvature tensors of $M^n$ and $\bar M^{n+1}$ respectively. The {\em equations of Gauss and Codazzi}, which directly follow from \eqref{fG}, { res\-pectively state}
\begin{align}
\label{eG}
& g(R (X,Y)Z,W) = g(R^M (X,Y)Z,W) + \varepsilon \left( h(X,Z)h(Y,W) - h(X,W)h(Y,Z) \right),\\[2pt]
\label{eC} &  g(R (X,Y)Z,\xi) = \varepsilon \left(
(\nabla^M h)(X,Y,Z) - (\nabla^M h)(Y,X,Z) \right),
\end{align}
where $X$, $Y$, $Z$ and $W$ are tangent to $M^n$. 

The hypersurface is said to have a {\em Codazzi second fundamental form} if {$\nabla^{M} h$} is totally symmetric. By equation~\eqref{eC}, this is equivalent to requiring that $R(X,Y)\xi=0$ for all vector fields $X,Y$ tangent to { $M^n$}. It is clear that totally geodesic and parallel hypersurfaces fall within the class of hypersurfaces with a Codazzi second fundamental form.

{We end this subsection recalling some other well known generalizations of totally geodesic hypersurfaces. The {\em mean curvature} of a hypersurface $M^n$ is defined by  
$$
{\rm tr}_{g_M} h = \frac{1}{n}\sum g_M^{ij}h_{ij},
$$
where $g_M$ denotes the pullback on $M^n$ of the metric of the ambient space and $g_M^{ij}$ are the components of $(g_M)^{-1}$ with respect to a given basis of vector fields tangent to $M^n$. The hypersurface is said to be {\em minimal} (respectively, {\em of constant mean curvature, {\em or} CMC)} if ${\rm tr}_{g_M} h =0$ (respectively,  ${\rm tr}_{g_M} h=\kappa$ for some real constant $\kappa$).}
%

\section{Totally geodesic hypersurfaces}
\setcounter{equation}{0}

Let $F: M \rightarrow (\bar M,g)$ denote the immersion of a hypersurface into a Siklos spacetime and $\xi$ a unit  normal vector field to the hypersurface. {We look for} some necessary algebraic conditions on the components of 
$\xi$ with respect to the frame $\{\pa1, \pa2, \pa3, \pa4\}$ on $\bar M$, in order for $M$ to have a Codazzi second fundamental form. {We prove the following.}

\begin{theorem}\label{t1}
Let $F: M \rightarrow \bar M$ be a hypersurface with a Codazzi second fundamental form and $\xi$ a unit normal vector field,  with $ g(\xi, \xi) =\varepsilon \in \{-1, 1\}$. Consider the coordinate vector fields $\{\pa i\}$ on 
$\bar M$ introduced above. Then, every point of $M$ has a neighborhood $U\subseteq M$ on which
$$
\xi=a \pa 1+\frac{x_3}{\beta} \cos \theta \pa 3 +\frac{x_3}{\beta} \sin \theta \pa 4
$$
for some { functions} $a, \theta :U\rightarrow \R$,  where $(H''_{33}-H''_{44})\sin(2\theta)=2x_3 H''_{34} \cos(2\theta)$. 
{In particular, $M$ is necessarily timelike.}
\end{theorem}

\begin{rem}
{\em Let us observe that in the statement of Theorem \ref{t1}, the quantities $x_3$ and $H''_{ij}$ mean $x_3\circ F$ and $H''_{ij}\circ F$, respectively. We shall use a similar notation from now on.}
\end{rem}

\begin{proof}
Consider $\xi= a\pa1+b\pa2+c\pa3+d\pa4$, for some functions $a, b, c, d: U\rightarrow \R$ { such that }
$g(\xi,\xi)=\varepsilon =\pm 1 \neq 0$. Then, the following vector fields are tangent to the hypersurface:
$$
\begin{array}{ll}
X_1=(a+Hb)\pa1-b\pa2, &\quad X_4=c\pa2-(a+Hb)\pa3,\\[6pt]
X_2=c\pa1-b\pa3, &\quad X_5=d\pa2-(a+Hb)\pa4,\\[6pt]
X_3=d\pa1-b\pa4, &\quad X_6=d\pa3-c\pa4.
\end{array}
$$
If $h$ is Codazzi, {then} equation \eqref{eC} yields that $R(X_i,X_j)\xi=0$ for every $i,j\in\{1,\ldots,6\}$. In particular, {we have}
\begin{align}
\label{123}
0&=R(X_1,X_2)\xi=-\frac{b^2}{2x_3}(cf_2+dx_3 H''_{34})\pa1+\frac{b^3}{2x^2_3}(f_1-2H)\pa3+\frac{b^3}{2}H''_{34}\pa4, \\[5pt]
\label{456}
0&=R(X_1,X_3)\xi=-\frac{b^2}{2x_3}(df_4+cx_3H''_{34})\pa1+\frac{b^3}{2}H''_{34}\pa3+\frac{b^3}{2x^2_3}(f_3-2H)\pa4,
\\[5pt]
\label{131415}
0&=R(X_1,X_6)\xi=b\big[(dc(f_2-f_4)+x_3 H''_{34}(d^2-c^2))\pa1
\\ \nonumber
&-\frac{b}{2x^2_3}(df_1-cx^2_3H''_{34}-2dH)\pa3
+\frac{b}{2x^2_3}(cf_3-dx^2_3H''_{34}-2cH)\pa4\big],
\\[5pt]
\label{28}
0&=R(X_4,X_5)\xi=(a+Hb)\big[-(dc(f_2-f_4)+x_3 H''_{34}(d^2-c^2))\pa1
\\ \nonumber
&+\frac{b}{2x^2_3}(df_1-cx^2_3H''_{34}-2dH)\pa3
-\frac{b}{2x^2_3}(cf_3-dx^2_3H''_{34}-2cH)\pa4\big],
\\[5pt]
\label{31}
0&=R(X_4,X_6)\xi=c\big[-(dc(f_2-f_4)+x_3 H''_{34}(d^2-c^2))\pa1
\\ \nonumber
&+\frac{b}{2x^2_3}(df_1-cx^2_3H''_{34}-2dH)\pa3
-\frac{b}{2x^2_3}(cf_3-dx^2_3H''_{34}-2cH)\pa4\big],
\\[5pt]
\label{34}
0&=R(X_5,X_6)\xi=d\big[-(dc(f_2-f_4)+x_3 H''_{34}(d^2-c^2))\pa1
\\ \nonumber
&+\frac{b}{2x^2_3}(df_1-cx^2_3H''_{34}-2dH)\pa3
-\frac{b}{2x^2_3}(cf_3-dx^2_3H''_{34}-2cH)\pa4\big].
\end{align}
We will treat separately two cases, depending on whether $b=0$ or $b\neq0$.

\medskip
\textbf{Case 1: $\boldsymbol{b}\mathbf{=0}$.} In this case, equations \eqref{28}--\eqref{34}, {together with the definitions given in \eqref{fi}, imply}
\begin{equation}\label{cond}
(H''_{33}-H''_{44})dc+H''_{34}(d^2-c^2)=0.
\end{equation}
Observe that $g(\xi, \xi)=\frac{\beta^2}{x_3^2}(c^2+d^2)=1${, so that $M$ is necessarily timelike}. The conclusion then follows setting $c=\frac{x_3}{\beta}\cos \theta$, $d=\frac{x_3}{\beta}\sin \theta$. Moreover, \eqref{cond} rewrites as
$$
{ \sin 2\theta (H''_{33}-H''_{44})=2\cos 2\theta H''_{34}.}
$$

\medskip
\textbf{Case 2: $\boldsymbol{b}\mathbf{\neq 0}$.} In this case, equation \eqref{123} implies that $H''_{34}=0$. Then, by \eqref{123} and \eqref{456} we get
$$
f_1=2H, \quad cf_2=0, \quad f_3=2H, \quad df_4=0,
$$
whence, by the definitions {\eqref{fi}} it follows at once
$$
x_3 H''_{33}=H'_3, \qquad 
x_3 H''_{44}=H'_3
$$
along the hypersurface. This means that the hypersurface $M$ is defined by some equations of the form:
\begin{equation}\label{Mcodcase2}
\begin{cases}
F_3\, H''_{33}(F_2,F_3,F_4)=H'_3(F_2,F_3,F_4), \\[6pt]
F_3 \, H''_{44}(F_2,F_3,F_4)=H'_3(F_2,F_3,F_4).
\end{cases}
\end{equation}
Let $(p_1,p_2,p_3,p_4)\in M$ and consider the curve defined by $\alpha(t)=(p_1+t,p_2,p_3,p_4)$, $t \in \R$. By \eqref{Mcodcase2} we deduce that $\alpha \subset M$ and $\alpha'(t_0)=\partial_{1_{(p_1+t_0,p_2,p_3,p_4)}}\in T_{\alpha(t_0)}M$ for every $t_0\in \R$, that is, $\pa 1$ is tangent to $M$. But this leads to a contradiction. In fact, as $\xi=a\pa1 +b\pa2 +c\pa3 +d\pa4$, we have $g(\xi,\pa1 )=\frac{\beta^2}{x_3^2}b \neq 0$, so that $\pa 1$ cannot be tangent to $M$. Therefore,  Codazzi hypersurfaces do not occur for $b\neq 0$.
\end{proof}

Consider now a hypersurface $M$ of a Siklos spacetime with a Codazzi second fundamental form, as described in Theorem~\ref{t1}. We shall discuss now the case where $\cos \theta =0$. In this case we will obtain in the following theorem the description of such hypersurfaces and in Theorem \ref{costhis0} of the totally geodesic ones. {On the other hand, totally geodesic hypersurface do not occur} in the case where $\cos\theta\neq 0$; therefore, we will later search there for parallel ones.

\begin{theorem}\label{costhis0}
Let $F: M \rightarrow \bar M$ denote a hypersurface of a Siklos spacetime with a Codazzi second fundamental form as described in Theorem~{\em\ref{t1}} and assume that $\cos\theta=0$. Then, there exist local coordinates $(u_1, u_2, u_3)$ on $M$ such that, up to isometries of the ambient space, the immersion is explicitly given by
{ \begin{align*}
F(u_1, u_2, u_3) = (u_1,u_2,u_3,f(u_2)),
\end{align*}}
where $f(u_2)$ satisfies the functional equation $H''_{34}(u_2,u_3,f(u_2))=0$. All such hypersurfaces are minimal.
\end{theorem}

\begin{proof}
Since $\xi=a \pa1+\frac{x_3}{\beta}\pa4$ for some function $a:U\rightarrow \R$, the following vector fields span the tangent space to $M$ at each point:
\begin{equation}\label{framecosthis0}
Y_1 = \pa1 ,\qquad Y_2 = { \frac{1}{\beta}\pa2-\frac{a}{x_3}\pa4}, \qquad Y_3 = \pa3.
\end{equation}
{ 
Using \eqref{framecosthis0} and \eqref{nabla}, a direct calculation gives
\begin{align}	\label{nablacosthis0}
\nabla_{Y_1}Y_1=&0, \quad
\nonumber
\nabla_{Y_2}Y_1=\frac{1}{\beta x_3}Y_3, \quad
\nabla_{Y_3}Y_1=-\frac{1}{x_3}Y_1, \\
\nonumber
\nabla_{Y_1}Y_2=&\frac{1}{\beta x_3}Y_3-\frac{\beta}{x_3^2}{Y_1(a)}(\xi-a{Y_1}),\\
\nonumber
\nabla_{Y_3}Y_2=&\left(\frac{H'_3}{2\beta}+\frac{a\beta}{x_3^2}\left(Y_3(a)-\frac{a}{x_3}\right)\right)Y_1-\frac{1}{x_3}Y_2-\frac{\beta}{x_3^2}\left(Y_3(a)-\frac{a}{x_3}\right)\xi,\\
\nabla_{Y_2}Y_2=&\left(\frac{1}{2\beta}\left(\frac{H'_2}{\beta}-\frac{a H'_4}{x_3}\right)+\frac{a\beta}{x_3^2}Y_2(a)\right)Y_1+
\frac{1}{x_3}\left(\frac{H}{\beta^2}-\frac{x_3 H'_3}{2\beta^2}-{a^2}\right)Y_3\\
\nonumber
&-\frac{\beta}{x_3}\left(\frac{H'_4}{2}+\frac{Y_2(a)}{x_3}\right)\xi,\\
\nonumber
  \nabla_{Y_1}Y_3=&-\frac{1}{x_3}Y_1, \quad
\nabla_{Y_2}Y_3=\frac{H'_3}{2\beta}Y_1-\frac{1}{x_3}Y_2, 
\quad
\nabla_{Y_3}Y_3=-\frac{1}{x_3}Y_3.
\end{align}}
From \eqref{nablacosthis0}, using the Gauss formula \eqref{fG} we get that the second fundamental form $h$ is completely determined by
\begin{equation}\label{hcosthis0}
h(Y_2,Y_2)={-\frac{\beta}{x_3}\left(\frac{H'_4}{2}+\frac{Y_2(a)}{x_3}\right)}, \qquad h(Y_i,Y_j)=0 \; \text{for all} \; (i,j)\neq(2,2),
\end{equation}
where we took into account the symmetry condition for $h$, which yields 
$$
Y_1(a)=0, \quad Y_3(a)=\frac{a}{x_3}.
$$
Moreover, the Levi-Civita connection on $M$ is completely determined by
{ 
\begin{align}\label{nablaMcosthis0}
\nabla^M_{Y_1}Y_1=&0, \quad
\nonumber
\nabla^M_{Y_2}Y_1=\frac{1}{\beta x_3}Y_3, \quad
\nabla^M_{Y_3}Y_1=-\frac{1}{x_3}Y_1, \\
\nonumber
\nabla^M_{Y_1}Y_2=&\frac{1}{\beta x_3}Y_3, \quad
\nabla_{Y_3}Y_2=\frac{H'_3}{2\beta}Y_1-\frac{1}{x_3}Y_2,\\
\nabla^M_{Y_2}Y_2=&\left(\frac{1}{2\beta}\left(\frac{H'_2}{\beta}-\frac{a H'_4}{x_3}\right)+\frac{a\beta}{x_3^2}Y_2(a)\right)Y_1+
\frac{1}{x_3}\left(\frac{H}{\beta^2}-\frac{x_3 H'_3}{2\beta^2}-{a^2}\right)Y_3,\\
\nonumber
\nabla^M_{Y_1}Y_3=&-\frac{1}{x_3}Y_1, \quad
\nabla^M_{Y_2}Y_3=\frac{H'_3}{2\beta}Y_1-\frac{1}{x_3}Y_2, 
\quad
\nabla^M_{Y_3}Y_3=-\frac{1}{x_3}Y_3.
\end{align}}
{ 
Thus, the vector fields $Y_1$, $Y_2$ and $Y_3$ are coordinate vector fields on $M$, so  we put
\begin{equation}\label{cooMcosthis0}
    \pa {u_1}=Y_1,\qquad 
    \pa {u_2}=Y_2, \qquad
    \pa {u_3}=Y_3.
\end{equation}}
With respect to the coordinates introduced above, the symmetry conditions for $h$ read 
\begin{equation}\label{conda}
\pa {u_1} a=0, \qquad \pa {u_3}(\ln a)=\frac{1}{x_3}.
\end{equation}

Denote now by $F: M \rightarrow \bar M, \; (u_1, u_2, u_3)\mapsto(F_1(u_1, u_2, u_3), \ldots, F_4(u_1, u_2, u_3))$ the immersion of the hypersurface in the local coordinates introduced above. Using \eqref{framecosthis0} and \eqref{cooMcosthis0} we obtain
\begin{equation} \label{eqpacosthis0}
\begin{array}{ll}
(\pa {u_1} F_1, \pa {u_1} F_2, \pa {u_1} F_3, \pa {u_1} F_4)=& (1,0,0,0),\\[6pt]
(\pa {u_2} F_1, \pa {u_2} F_2, \pa {u_2} F_3, \pa {u_2} F_4)=& (0,\frac{1}{\beta},0,-\frac{a}{F_3}),\\[6pt]
(\pa {u_3} F_1, \pa {u_3} F_2, \pa {u_3} F_3, \pa {u_3} F_4)=& {(0,0,1,0)}.
\end{array}
\end{equation}
Integrating \eqref{eqpacosthis0} we get
$$
\begin{array}{l}
F_1= u_1+c_1,\quad 
F_2= \frac{u_2}{\beta}+c_2,\quad F_3= u_3+c_3,\quad F_4= f(u_2),
\end{array}
$$
for some real constants $c_i, i=1,2,3$ and a function $f$, which depends only on $u_2$ and satisfies $f'(u_2)=-\frac{a}{F_3}=-\frac{a}{u_3+c_3}$. It then follows from \eqref{conda} that $a(u_2,u_3)=A(u_2)(u_3+c_3)$ for an arbitrary function $A$. As we restricted to the case $\cos \theta=0$, from Theorem \ref{t1} we derive the relation $H''_{34}\circ F=0$ and, after a reparametrization, we obtain the immersion given in the statement.

{Finally, by \eqref{hcosthis0}, the second fundamental form of such hypersurfaces is completely determined by the only possibly non-vanishing component
$$
h(Y_2,Y_2)={-\frac{\beta}{x_3}\left(A'(u_2)+\frac{H'_4}{2}\right)}.
$$
Consequently,  for all hypersurfaces of Siklos spacetimes described above we get ${\rm tr}_{g_M} h=g_M^{22}h_{22}=0$ and so, they  are minimal.}
\end{proof}

{
\begin{rem}{\em
A hypersurface as described in Theorem~\ref{costhis0} may be interpreted as a sort of cylinder over the curve with equation $x_4=f(x_2)$ in a surface of the form $x_1=\mathrm{constant}$, $x_3=\mathrm{constant}$.}
\end{rem}}

{We now classify totally geodesic hypersurfaces of Siklos spacetimes.}

\begin{theorem}\label{main}

Let $F: M \rightarrow \bar M$ denote a totally geodesic hypersurface of a Siklos spacetime. Then, there exist local coordinates $(u_1, u_2, u_3)$ on $M$ such that, up to isometries of the ambient space, the immersion is explicitly given by
{ \begin{align*}
F(u_1, u_2, u_3) = (u_1,u_2,u_3,f(u_2)),
\end{align*}}
where $f(u_2)$ satisfies the functional equation $H''_{34}(u_2,u_3,f(u_2))=0$ and the ODE 
$$
2f''(u_2)={H'_4(u_2, u_3, f(u_2))}.
$$
\end{theorem}

{\begin{proof}
We start from the description given in 
Theorem~\ref{t1} of hypersurfaces with $h$ Codazzi and consider first the case where $\cos\theta=0$. 

\medskip
\textbf{Case 1: $\boldsymbol{\cos\theta}\mathbf{=0}$.}
Hence, the hypersurfaces are as described in Theorem~\ref{costhis0} and we have the tangent coordinate vector fields explicitly given by
$$
F_{u_1}=\pa 1, \quad F_{u_2}=\pa 2+f'(u_2)\pa 4, \quad F_{u_3}=\pa 3
$$
and the unit normal vector field on $M$ expressed by
$$
\xi=f'(u_2)\pa 1-\frac{\beta}{x_3}\pa 4.
$$
By \eqref{nablahis0} and \eqref{nablacosthis0} it follows that the only components of the Levi-Civita connection wich take part in the condition for parallelism are the following:
\begin{align}\label{nablaF}
\nabla_{F_{u_2}}F_{u_2}&=\left(\frac{H'_2}{2}+f'\frac{H'_4}{2}+f'f''\right) F_{u_1}+\left(\frac{H}{x_3}-\frac{H'_3}{2}+\frac{(f')^2}{x_3}\right)F_{u_3}-\frac{\beta}{x_3}\left(f''-\frac{H'_4}{2}\right)\xi,\\
\nonumber
\nabla_{F_{u_2}}F_{u_3}&=\nabla^M_{F_{u_3}}F_{u_2}=\frac{H'_3}{2}F_{u_1}-\frac{1}{x_3}F_{u_2}.
\end{align}
Using \eqref{nablaF}, the condition $H''_{34}=0$ and the fact that $f$ depends only on $u_2$, a direct calculation yields that the immersion is parallel if and only if
$$
\pa {u_2}\left(f''-\frac{H'_4}{2}\right)=0, \qquad \frac{\beta}{x_3}\left(f''-\frac{H'_4}{2}\right)=0.
$$
On the other hand,
$$
h(F_{u_2},F_{u_2})=-\frac{\beta}{x_3}\left(f''-\frac{H'_4}{2}\right)=0
$$
and so, the immersion is parallel if and only if it is totally geodesic. Thus, for $\cos\theta=0$ there are not examples of proper parallel hypersurfaces, and the totally geodesic ones are completely determined by condition $f''=\frac{H'_4}{2}$. 

\medskip
\textbf{Case 2: $\boldsymbol{\cos\theta}\mathbf{\neq 0}$.}
Let now $F: M \rightarrow \bar M$ denote a hypersurface as in Theorem~\ref{t1} with $\cos \theta \neq 0$. Then, $\xi=a\pa1+\frac{x_3}{\beta} \cos \theta \pa 3 +\frac{x_3}{\beta} \sin \theta \pa 4$ and vector fields
\begin{equation}\label{frameI}
Y_1 = \pa1,\quad Y_2 = \frac{x_3}{\beta} \sin \theta \pa 3 -\frac{x_3}{\beta}\cos \theta \pa4, \quad Y_3 = \frac{x_3}{\beta}\cos \theta \pa2-a\pa3
\end{equation}
span the tangent space to $M$ at every point.

Using \eqref{frameI}, \eqref{nabla}and the Gauss formula, a long but straightforward calculation yields
\begin{align*}
\nabla^M_{Y_1}Y_1=& 0, \quad \nabla^M_{Y_2}Y_1=-\frac{\sin \theta}{\beta}Y_1, \quad \nabla^M_{Y_3}Y_1=\sin^2\theta \frac{a}{x_3} Y_1+\frac{\sin 2\theta}{2x_3}Y_2, \\
\nabla^M_{Y_1}Y_2=&-\left( aY_1(\theta)+\frac{\sin \theta}{\beta}\right)Y_1, \quad \nabla^M_{Y_2}Y_2=-a\left(Y_2(\theta)+\frac{\cos\theta}{\beta}\right)Y_1,\\
\nabla^M_{Y_3}Y_2=&\left(a^2\frac{\sin 2\theta}{2x_3}-aY_3(\theta)+\frac{x_3^2}{\beta^2}\frac{\cos\theta}{2}(\sin\theta H'_3-\cos\theta H'_4)\right)Y_1-a\frac{\sin^2\theta}{x_3} Y_2-\frac{\sin\theta}{\beta}Y_3,\\
\nabla^M_{Y_1}Y_3=&\frac{a}{x_3}\left(\sin^2\theta+\cos^2\theta \beta Y_1\left(\frac{a}{\cos\theta}\right)\right)Y_1+\frac{\sin 2\theta}{2 x_3}\left(1-\beta Y_1\left(\frac{a}{\cos\theta}\right)\right)Y_2+\frac{Y_1(\cos\theta)}{\cos\theta}Y_3,\\
\nabla^M_{Y_2}Y_3=&\left(\frac{x_3^2}{\beta^2}\frac{\cos\theta}{2}(\sin\theta H'_3-\cos\theta H'_4)+a\frac{\beta}{x_3}\cos^2\theta Y_2\left(\frac{a}{\cos\theta}\right)\right)Y_1\\
&+\left(\frac{a}{x_3}-\frac{\beta}{x_3}\frac{\sin2 \theta}{2} Y_2\left(\frac{a}{\cos\theta}\right)\right) Y_2-\frac{Y_2(\cos\theta)}{\cos\theta}Y_3,\\
\nabla^M_{Y_3}Y_3=& \Bigl(\frac{x_3^2}{\beta^2}\frac{\cos^2\theta}{2}H'_2-a\frac{x_3}{\beta}\frac{\cos\theta}{2}(2-\cos^2\theta)H'_3+a\frac{\beta}{x_3}\cos^2\theta Y_3\left(\frac{a}{\cos\theta}\right)\\
&+a\frac{x_3}{\beta}\frac{\cos^2\theta}{2}\sin\theta H'_4-a\frac{\cos^3\theta}{\beta}H \Bigr)Y_1+\Bigl(-\frac{x_3}{\beta}\frac{\cos^2\theta}{2}(\sin\theta H'_3-\cos\theta H'_4)\\
&+\cos^2\theta \sin\theta \frac{H}{\beta}-\frac{\beta}{x_3}\frac{\sin 2\theta}{2}Y_3\left(\frac{a}{\cos\theta}\right)\Bigr)Y_2+\left(\frac{a}{x_3}+\frac{Y_3(\cos\theta)}{\cos\theta}\right)Y_3
\end{align*}
and the second fundamental form is determined by

$$
\begin{array}{lll}
h(Y_1,Y_1)=0,
&\quad
h(Y_1,Y_2)=Y_1(\theta),
\\[6pt]
h(Y_1,Y_3)=\frac{\cos^2\theta}{x_3}\left(1-\beta Y_1\left(\frac{a}{\cos\theta}\right)\right),
&\quad
h(Y_2,Y_1)=0,
\\[6pt]
h(Y_2,Y_2)=Y_2(\theta)+\frac{\cos\theta}{\beta},
&\quad 
h(Y_2,Y_3)=-\frac{\beta}{x_3}\cos^2 \theta Y_2\left(\frac{a}{\cos\theta}\right),
\\[6pt]
h(Y_3,Y_1)=\frac{\cos^2\theta}{x_3},
&\quad
h(Y_3,Y_2)=Y_3(\theta)-a\frac{\sin\theta\cos\theta}{x_3},
\end{array}
$$
$$
h(Y_3,Y_3)=\cos^3\theta\frac{H}{\beta}-\frac{x_3}{\beta}\frac{\cos^2\theta}{2}(\cos\theta H'_3+\sin\theta H'_4)-\frac{\beta}{x_3}\cos^2\theta Y_3\left(\frac{a}{\cos\theta}\right).
$$
Next, the symmetry of $h$ yields
$$
Y_1(\theta)=0, \qquad Y_1(a)=0,\qquad  Y_3(\theta)=a\frac{\sin\theta\cos\theta}{x_3}-\frac{\beta}{x_3}\cos^2 \theta Y_2\left(\frac{a}{\cos\theta}\right).
$$
Observe that as $h(Y_1,Y_3)=h(Y_3,Y_1)=\frac{\cos^2\theta}{x_3}$, the hypersurface $M$ cannot be totally geodesic for $\cos\theta \neq 0$. 
\end{proof}

{\begin{rem}{\em 
As we proved in the above Theorem \ref{main}, the condition $\cos \theta =0$ is necessary in order to get totally geodesic hypersurfaces. Consequently, by Theorem~\ref{t1} we have $H''_{34}\circ F=0$. Although it is not possible to decide whether a totally geodesic hypersurface exists for a general defining function $H$, we shall prove that there are some classes of Siklos metrics, defined by some particular functions $H$, for which we can give explicit immersions of totally geodesic hypersurfaces. 
}\end{rem}}

We may observe that the condition we found for the defining condition $H$ in Theorems~\ref{costhis0} and~\ref{main} naturally occurs in the study of geometric properties of Siklos metrics \eqref{Smet}. In fact, $H''_{34}=0$ is satisfied by all (locally) conformally flat Siklos metrics (see for example \cite{CcflatS}).
The restriction we found leaves us with a very large class of spacetimes, not necessarily homogeneous. In fact,  $H''_{34}=0$ corresponds to requiring that
$$
{
H=L(x_2,x_3)+G(x_2,x_4)},
$$
for two arbitrary functions ${L},G$ of two variables \cite{ClargeS}. By Proposition~\ref{costhis0}, all such Siklos spacetimes admit minimal hypersurfaces with a Codazzi second fundamental form.

{
In addition, the condition for totally geodesic hypersurfaces given in Theorem~\ref{main} now reads:
$$
f''(u_2)=\tilde{G}(u_2),
$$
where $\tilde{G}(u_2)=(G'_4\circ F)(u_2)=\left(\dfrac{\partial}{\partial x_4}(G)\circ F\right)(u_2)$.

Moreover, we observe that the special subclass determined by condition $H'_4=0$, {identifies in Siklos paper \cite{Si}} the class of Siklos metrics which, besides $\partial_1$ (which is a Killing vector field for all Siklos metrics), admits at least two further linearly independent Killing vector fields, namely, $\partial_4$ and $x_4 \partial_1 -x_2\partial_4$. The condition $H'_{4}=0$ is clearly equivalent to requiring that 
$$
H=H(x_2,x_3).
$$
For any of such defining functions, the corresponding Siklos spacetime admits {totally geodesic} hypersurfaces as described in Theorem~\ref{main}. In particular, for the totally geodesic case, the function $f$ given therein is a polynomial function of degree at most $1$ of $u_2$. More explicitly, the following result gives the expression for such hypersurfaces in the case of $H'_4=0$.

\begin{cor}\label{H'_4=0}
Let $(\bar M,g)$ denote an arbitrary Siklos spacetime, with the Lorentzian metric $g$ described by \eqref{Smet} with respect to global coordinates $(x_1, x_2, x_3, x_4)$ and depending on an arbitrary defining smooth function $H=H(x_2,x_3)$. Then, $(\bar M,g)$ always admits totally geodesic hypersurfaces which, up to isometries of the ambient space, are  explicitly described by
$F: M \rightarrow (\bar M,g)$ with
\begin{align}
F(u_1,u_2,u_3) = (u_1,u_2,u_3,\lambda u_2+\mu), \label{immparH'_4=0}
\end{align}
where $\lambda$ and $\mu$ are real constants. {Therefore, the image of $M$ is a hyperplane of the ambient space 
$(\bar M,g)$, of equation $x_4=\lambda x_2 + \mu$.}
\end{cor}
}}

\section{Proper parallel Hypersurfaces}
\setcounter{equation}{0}

Let $F: M \rightarrow \bar M$ denote again a hypersurface of a Sikos spacetime.  We now require that the immersion of $M$ is parallel but not totally geodesic. As we proved in the previous Section, this implies that $M$ is as described in 
Theorem~\ref{t1} with $\cos \theta \neq 0$, and we already derived in such a case the equations for $\nabla^M$ 
and $h$. 

By these equations, we now find 
\begin{align}
\label{p1} 0&=(\nabla^M _{Y_2} h) (Y_1,Y_3)= -\frac{\sin\theta \cos\theta}{x_3}Y_2(\theta),\\[7pt]
\label{p2} 0&=(\nabla^M _{Y_2} h) (Y_3,Y_2)=Y_2(Y_3(\theta))=\frac{\beta}{x_3}\sin\theta Y_3(\theta),\\[7pt]
\label{p3} 0&=(\nabla^M _{Y_3} h) (Y_3,Y_2)=-\frac{\sin\theta \cos\theta}{x_3}\left(\frac{a^2}{x_3}+Y_3(a)+\frac{x_3^2}{\beta^2}\frac{\cos\theta}{2}(\cos\theta H'_3+\sin\theta H'_4)\right).
\end{align}
Equations \eqref{p1} and \eqref{p2}, together with the symmetry condition $Y_1(\theta)=0$, imply that  $\theta$ is necessarily constant. {Therefore, from Theorem \ref{t1} we have at once in full generality the following necessary condition for the existence of parallel hypersurfaces.}

\begin{prop}\label{mainp}
Let $(\bar M,g)$ denote an arbitrary Siklos spacetime, with the Lorentzian metric $g$ described by \eqref{Smet} with respect to global coordinates $(x_1, x_2, x_3, x_4)$ and depending on an arbitrary defining smooth function $H=H(x_2,x_3,x_4)$. {If $(\bar M,g)$ admits parallel hypersurfaces, then there exists a real} constant $\theta$, such that the defining function $H=H(x_2,x_3,x_4)$ satisfies {the partial differential equation}
\begin{equation}\label{cpara}
{  \sin (2\theta) (H''_{33}-H''_{44})=2\cos (2\theta)H''_{34}}
\end{equation}
along the hypersurface.
\end{prop}

We may observe that condition \eqref{cpara} generalizes (locally) conformally flat Siklos metrics, completely characterized by  $H''_{33}-H''_{44}=H''_{34}=0$ \cite{CcflatS}.

In the previous Section we derived all formulas which permit to express parallelism of a  hypersurface with a Codazzi second fundamental form in  an arbitrary Siklos spacetime. However, the condition for parallelism leads in some cases to partial differential equations which are extremely difficult to treat. This should not come as a surprise, taking into account that here we are not just dealing with an example of ambient space, but with the whole class of Siklos metrics, so that the defining function $H=H(x_2,x_3,x_4)$ plays a crucial role when expressing parallelism.

In the remaining part of this Section, we first give a complete description of some classes of proper parallel hypersurfaces in Sikos spacetimes. We then specify the defining function $H$ to correspond to some well known homogeneous spacetimes and in this case we achieve a complete description of their parallel hypersurfaces.

We start with any Siklos metric, as described by an arbitrary defining function $H=H(x_2,x_3,x_4)$. A proper parallel hypersurface, having $h$ Codazzi, is described as in Theorem  \ref{t1}. We restrict here to the case where $a=0$.

\begin{theorem}\label{42}
Let $F: M \rightarrow \bar M$ be a proper parallel hypersurface of a Siklos homogeneous spacetime. Assume that $a=0$ in the description obtained in Theorem~{\em\ref{t1}}. Then, there exist local coordinates $(u_1, u_2, u_3)$ on $M$, such that up to isometries of the ambient space, the immersion is given by  one of the following expressions:

\smallskip
$\bullet$ either there exists a real constant $\theta$, such that the defining function $H=H(x_2,x_3,x_4)$ satisfies {the partial differential equation} $H'_3=-\tan\theta H'_4$ along the hypersurface and
\begin{itemize}
\vspace{5pt}\item[1.] $F(u_1, u_2, u_3)=\left(u_1,u_3,\rho e^{u_2},-\cot \theta \rho e^{u_2}+ C\right)
$ for some real constants $C$, $\theta$ and $\rho \neq 0$, or
\end{itemize}

\smallskip
$\bullet$   $H''_{32}=H''_{34}=0$ along the hypersurface and
\begin{itemize}
\vspace{5pt}\item[2.] $F(u_1, u_2, u_3)=(u_1,u_3,C,u_2)$, where $C\neq 0$ is a real constant.
\end{itemize}
{Moreover, both classes of parallel hypersurfaces are of constant mean curvature (CMC).}
\end{theorem}

\begin{proof}
{Observe that by Theorem~\ref{main} we have $\cos\theta \neq 0$, as we are looking for proper parallel hypersurfaces. Thus, by  equation~\eqref{p3} it follows that either $\sin \theta =0$ or 
$\cos\theta H'_3+\sin\theta H'_4=0$ along the hypersurface, that is,
$$
H'_3=-\tan\theta H'_4.
$$
We will treat the case $\sin\theta=0$ later. We now} look for a system of local coordinates $(u_1, u_2, u_3)$ on $M$, such that
\begin{equation}\label{cooMa12}
    \pa {u_1}=Y_1,\qquad 
    \pa {u_2}=Y_2, \qquad
    \pa {u_3}=\alpha Y_2+ \gamma Y_3,
\end{equation}
for some smooth functions $\alpha, \gamma$ on  $M$. Requiring that $[ \pa {u_2},\pa {u_3}]=0$, we get
\begin{align*}
\begin{cases}
   Y_2(\alpha)=0,\\[5pt]
   Y_2(\gamma)=-\gamma\frac{\sin\theta}{\beta} ,
\end{cases}
\end{align*}
which admits as a solution 
$$
\alpha=0, \quad \gamma=e^{-\frac{\sin\theta}{\beta}u_2}.
$$

Denote now by $F: M \rightarrow \bar M, \; (u_1, u_2, u_3)\mapsto(F_1(u_1,u_2,u_3), \ldots, F_4(u_1,u_2,u_3))$ the immersion of the hypersurface in the local coordinates introduced above. By~\eqref{frameI}, we obtain
\begin{equation} \label{eqpaA12}
\begin{array}{ll}
(\pa {u_1} F_1, \pa {u_1} F_2, \pa {u_1} F_3, \pa {u_1} F_4)=& (1,0,0,0),\\[6pt]
(\pa {u_2} F_1, \pa {u_2} F_2, \pa {u_2} F_3, \pa {u_2} F_4)=& \frac{F_3}{\beta}{(0,0,\sin\theta,-\cos\theta)},\\[6pt]
(\pa {u_3} F_1, \pa {u_3} F_2, \pa {u_3} F_3, \pa {u_3} F_4)=& e^{-\frac{\sin\theta}{\beta}u_2}\frac{F_3}{\beta}(0,{\cos\theta},0,0).
\end{array}
\end{equation}
Integrating \eqref{eqpaA12} we find
$$
 F_1= u_1+c_1,\quad F_2= c_3\frac{\cos\theta}{\beta}u_3+c_2,\quad F_3=c_3e^{\frac{\sin\theta}{\beta}u_2},\quad F_4= -c_3\frac{\cot\theta}{\beta}e^{\frac{\sin\theta}{\beta}u_2}+ c_4,
$$
for some real constants $c_1$, $c_2$, $c_3$ and $c_4$. 
After a reparametrization and applying isometries of the ambient space one obtains the following description in local coordinates of these proper parallel hypersurfaces:
$$
F(u_1, u_2, u_3)=\left(u_1,u_3,\rho e^{u_2},-\cot \theta \rho e^{u_2}+ C\right)
$$
for some real constants $C$, $\theta$ and $\rho\neq 0$.

\medskip
{ We now consider the case $\sin\theta=0$.} Then, from Theorem~\ref{mainp} we deduce $H''_{34}=0$. Moreover, in this case vector fields
$$
\pa {u_1}=Y_1 = \pa1,\qquad \pa {u_2}=Y_2 = -\frac{x_3}{\beta} \pa4, \qquad \pa {u_3}=Y_3 = \frac{x_3}{\beta}\pa2
$$
are coordinates vector fields on $M$ and the only non-vanishing components of second fundamental form are
$$
h(Y_1,Y_3)=\frac{1}{x_3},\qquad
h(Y_2,Y_2)=\frac{1}{\beta},\qquad
h(Y_3,Y_3)=\frac{H}{\beta}-\frac{x_3}{\beta}\frac{H'_3}{2}.
$$
It is easy to see that, in this case, $M$ is flat. Requiring the immersion of $M$ to be parallel, we obtain
$$
0=(\nabla^M _{Y_3} h) (Y_3,Y_3)= -\frac{x_3^2}{\beta^2}\frac{H''_{32}}{2},
$$
that is $H''_{32}=0$.

Denote now by $F: M \rightarrow \bar M, \; (u_1, u_2, u_3)\mapsto(F_1(u_1,u_2,u_3), \ldots, F_4(u_1,u_2,u_3))$ the immersion of the hypersurface in the local coordinates introduced above. Then, we obtain
\begin{equation} \label{eqpaA22}
\begin{array}{ll}
(\pa {u_1} F_1, \pa {u_1} F_2, \pa {u_1} F_3, \pa {u_1} F_4)=& (1,0,0,0),\\[6pt]
(\pa {u_2} F_1, \pa {u_2} F_2, \pa {u_2} F_3, \pa {u_2} F_4)=& -\frac{F_3}{\beta}{(0,0,0,1)},\\[6pt]
(\pa {u_3} F_1, \pa {u_3} F_2, \pa {u_3} F_3, \pa {u_3} F_4)=& \frac{F_3}{\beta}(0,1,0,0).
\end{array}
\end{equation}
and, by integration,
$$
 F_1= u_1+c_1,\quad F_2=\frac{c_3}{\beta}u_3+c_2,\quad F_3=c_3,\quad F_4=-\frac{c_3}{\beta}u_2+ c_4,
$$
for some real constants $c_i, i=1,\ldots ,4$. After a reparametrization we conclude that $F(u_1, u_2, u_3)=(u_1,u_3,C,u_2)$, which is a sort of hyperplane of the form $x_3=C$ for some real constant $C \neq 0$.

{ We conclude the investigation of these hypersurfaces by computing their mean curvature. The second fundamental form of a generic Codazzi hypersurface of a Siklos spacetime with $\cos \theta =0$ was completely described within the proof of Theorem~\ref{main}, with respect to the basis $\{Y_1,Y_2,Y_3\}$ of tangent vector fields on $M$ given in \eqref{frameI}. For the parallel hypersurfaces we classified above we have the additional conditions that $\theta$ is a real constant  and $a=0$.  Denoted by $(g_M^{ij})=(g_M^{-1}(Y_i,Y_j))$ the components of $g_M^{-1}$ with respect to $\{Y_1,Y_2,Y_3\}$, a direct calculation yields that the mean curvature of $M$ is given by 
$$
\begin{array}{rcl}
{\rm tr}_{g_M} h &=& \frac{1}{3}\sum g_M^{ij}h_{ij} = \frac{1}{3}\left( 2 g_M^{13}h_{13}+g_M^{22}h_{22}\right)\\[4pt]
&=&\frac{1}{3} \left( 2\frac{x_3}{\beta \cos\theta} \cdot 
\frac{\cos^2 \theta}{x_3} +1 \cdot \frac{\cos \theta}{\beta}  \right)=\frac{1}{\beta}  \cos\theta .
\end{array}
$$
(In case 2., as $\sin\theta=0$, the same conclusion holds with $\cos\theta=1$.) Therefore, the above proper parallel  hypersurfaces of Siklos spacetimes are CMC.}
\end{proof}

{We end this paper with the complete classification of proper parallel hypersurfaces in} a special case of great interest, namely,  {\em homogeneous} Siklos spacetimes corresponding to the defining function
$$
H=\varepsilon x_3^{2k}, 
$$
where $\varepsilon=\pm 1$ and $k$ is a real constant \cite{Si}.  This class includes:
\begin{itemize}
\item the pure radiation solution of Petrov type $N$ with a $G_6$ isometry group, first described by Defrise \cite{Defr}, obtained for $k=-1$;
\item the {\em Kaigorodov spacetime}, which is the only homogeneous type-$N$ solution of the Einstein vacuum field equations with $\Lambda \neq 0$ (\cite{Kai},\cite{Pod},\cite[Theorem~12.5]{Ste}) and is obtained for $k=3/2$;
\item the homogeneous solution to Einstein-Maxwell equations investigated first by Ozsv\'{a}th \cite{Ozs} for $k=2$.
\end{itemize}
{The classification of totally geodesic hypersurfaces of these homogeneous spacetimes follows from the more general result described in Corollary \ref{H'_4=0}. So, they are of the form
$$
F(u_1,u_2,u_3) = (u_1,u_2,u_3,\lambda u_2+\mu),
$$ 
for some real constants $\lambda,\mu$. With regard to proper parallel hypersurfaces, we have the following.}

\begin{theorem}
Let $F: M \rightarrow \bar M$ be a proper parallel hypersurface of a Siklos homogeneous spacetime, {described 
by~\eqref{Smet} with respect to global coordinates $(x_1, x_2, x_3, x_4)$ and determined by a defining function 
$H=\varepsilon x_3^{2k}$.} Then there exist local coordinates $(u_1, u_2, u_3)$ on $M$, such that up to isometries of the ambient space, the immersion is given by one of the following cases:

{\smallskip
$\bullet$ For $k=0$:
\begin{itemize}
\vspace{5pt}\item[(1a)]
$F(u_1, u_2, u_3)=\left(u_1,u_3,C e^{u_2},D e^{u_2}\right)$
for some real constants $C$ and $D$;
\vspace{5pt}\item[(1b)] $F(u_1,u_2,u_3)= \left(u_1,C u_3,D e^{u_2}, E e^{u_2}-u_3\right)$, for some real constants $C$, $D$ and $E$;
\end{itemize}}

\smallskip
$\bullet$ For $k=\frac{1}{2}$:
\begin{itemize}
\vspace{5pt}\item[(2)] $F(u_1,u_2,u_3)= \left(u_1,-\varepsilon\frac{2\beta}{C \cos\theta}\sqrt{-\varepsilon\left(\frac{C \cos\theta}{\beta}\right)^2u_3+D}, \beta u_2, -\cot\theta\left(u_2+u_3\right)\right)$, for some real constants $\theta,C,D$;
\end{itemize}

\smallskip
$\bullet$ For $k=1$ and $\varepsilon=1$:
\begin{itemize}
\vspace{5pt}\item[(3)] $F(u_1, u_2, u_3)=\left(u_1, u_3,\rho e^{-u_3}, -u_2 e^{-u_3}\right)$,
for some real constant $\rho$;
\end{itemize}

\smallskip
$\bullet$ For any $k\neq 1$ and $\varepsilon=1$:
\begin{itemize}
\vspace{5pt}\item[(4)] $F(u_1,u_2,u_3)=\left(u_1,\dfrac{u_3^{1-k}}{k-1},u_3, -\dfrac{u_2 u_3}{\beta}\right)$;
\end{itemize}

$\bullet$ For $k=-1$:

\begin{itemize}
\vspace{5pt}\item[(5)] 
${F(u_1, u_2, u_3)=\left(u_1,-\dfrac{u_3}{\beta} ,u_3, -\dfrac{u_2 u_3}{\beta} \right)}$;
\end{itemize}

\smallskip
$\bullet$ For any $k$:
\begin{itemize}
\vspace{5pt}\item[(6)] $F(u_1, u_2, u_3)=(u_1,u_3,C,u_2)$, where $C \neq 0$ is a real constant.
\end{itemize}
\end{theorem}

\begin{proof}

{As $H$ only depends on $x_3$,} condition \eqref{cpara} becomes
\begin{equation}\label{cparax_3}
\cos\theta \sin \theta H''_{33}=0.
\end{equation}
We already know from Theorem~\ref{main} that {the case for $\cos\theta=0$ does not yield any proper parallel hypersurfaces. Moreover, parallel hypersurfaces with $a=0$ were already classified in general in Theorem~\ref{42}.  
For a defining function of the special form $H=\varepsilon x_3^{2k}$, Case~1. of Theorem~\ref{42} easily yields $k=0$ and leads to case (1a) here.  Case~2. of Theorem~\ref{42} provides solutions for all values of $k$ and corresponds here to case (6).} 

So, we {are left to investigate} the cases with $a \neq 0$. We shall consider separately the cases $\sin\theta\neq 0$ and $\sin\theta=0$

\medskip
\textbf{Case 1: $\boldsymbol{\sin\theta}\mathbf{\neq 0}$.}
{
In this case, from \eqref{cparax_3} we get $H''_{33}=0$. Then, $H(x_3)=\varepsilon x_3$, that is, $k=\frac{1}{2}$ (excluding the trivial case $k=0$, which we will treat later). From \eqref{p3} we get
\begin{equation}\label{p3n}
Y_3(a)=-\frac{a^2}{x_3}-\varepsilon\frac{x_3^2}{\beta^2}\frac{\cos^2\theta}{2}.
\end{equation}
Next, since $\theta$ is constant, we can rewrite the last symmetry condition as
$$
Y_2(a)=a\frac{\sin\theta}{\beta}.
$$}

We now look for a system of local coordinates $(u_1, u_2, u_3)$ on $M$, such that
\begin{equation}\label{cooMa11}
    \pa {u_1}=Y_1,\qquad 
    \pa {u_2}=Y_2, \qquad
    \pa {u_3}=\alpha Y_2+ \gamma Y_3,
\end{equation}
for some smooth functions $\alpha,\gamma$ on  $M$. Requiring that $[ \pa {u_2},\pa {u_3}]=0$, we get
\begin{align*}
\begin{cases}
   Y_2(\alpha)=-\gamma\frac{a}{x_3} ,\\[5pt]
   Y_2(\gamma)=-\gamma\frac{\sin\theta}{\beta} .
\end{cases}
\end{align*}
Observe that we only need one solution for $\alpha$ and $\gamma$ in the system above in order to find a coordinate system 
$(u_1, u_2, u_3)$ on the surface $M$. So, we take
$$
\alpha=\frac{1}{\sin\theta}\frac{\beta}{x_3}, \quad \gamma=\frac{1}{a}.
$$
With respect to the coordinates introduced above, the symmetry condition reads
$$
\pa {u_1}(a)=0, \quad \pa {u_2} (a)=a\frac{\sin\theta}{\beta},
$$
so that there exists a positive smooth function {$\psi=\psi(u_3)$} on $M$ such that
$$
a(u_2,u_3)=\psi(u_3)e^{\frac{\sin\theta}{\beta}u_2}.
$$

Denote now by $F: M \rightarrow \bar M, \; (u_1, u_2, u_3)\mapsto(F_1(u_1,u_2,u_3), \ldots, F_4(u_1,u_2,u_3))$ the immersion of the hypersurface in the local coordinates introduced above. By \eqref{frameI}, we obtain
\begin{equation} \label{eqpaA11}
\begin{array}{ll}
(\pa {u_1} F_1, \pa {u_1} F_2, \pa {u_1} F_3, \pa {u_1} F_4)=& (1,0,0,0),\\[6pt]
(\pa {u_2} F_1, \pa {u_2} F_2, \pa {u_2} F_3, \pa {u_2} F_4)=& \frac{F_3}{\beta}{(0,0,\sin\theta,-\cos\theta)},\\[6pt]
(\pa {u_3} F_1, \pa {u_3} F_2, \pa {u_3} F_3, \pa {u_3} F_4)=& (0,\frac{F_3}{\beta}\frac{\cos\theta}{a},0,-\cot\theta).
\end{array}
\end{equation}
Integrating \eqref{eqpaA11} we find
$$
\begin{array}{ll}
F_1= u_1+c_1, &\quad F_2= \frac{c_3 \cos\theta}{\beta}\displaystyle\int_{c_2}^{u_3}{\psi(s)^{-1}}ds,\\[10pt] F_3=c_3e^{\frac{\sin\theta}{\beta}u_2},&\quad F_4= -\cot\theta\left(\frac{c_3}{\beta}e^{\frac{\sin\theta}{\beta}u_2}+u_3\right)+ c_4
 \end{array}
$$
for some real constants $c_i, i=1,\ldots ,4$. Consequently, condition \eqref{p3n} rewrites as
$$
(\psi^2)'=-\varepsilon\left(\frac{c_3 \cos\theta}{\beta}\right)^2,
$$
which gives $\psi(u_3)=\sqrt{-\varepsilon\left(\frac{c_3 \cos\theta}{\beta}\right)^2u_3+c_5}$, with $c_5$ a real constant.

Then, substituting this result in the parametrization and applying isometries of the ambient space one obtains the case (2) of the theorem.

We  now  consider the trivial case where $k=0$. Let us observe that \eqref{p3n} rewrites as
$$
Y_3(a)=-\frac{a^2}{x_3}.
$$
Then, using the coordinates found in \eqref{cooMa11} and the consequent expression for the function $a$, namely, $a=\psi(u_3)e^{\frac{\sin\theta}{\beta}u_2}$, the condition above reduces to $\psi'=0$, that is $a=R e^{\frac{\sin\theta}{\beta}u_2}$ for some real constant $R$.

Substituting this result in the correspondent parametrization and applying isometries of the ambient space and some reparametrizations, one obtains the case (1b) of the theorem.

\medskip
\textbf{Case 2: $\boldsymbol{\sin\theta}\mathbf{=0}$.}

In this case, \eqref{frameI} becomes
\begin{equation}\label{frameA2}
Y_1 = \pa1,\quad Y_2 = -\frac{x_3}{\beta} \pa4, \quad Y_3 = \frac{x_3}{\beta}\pa2-a\pa3.
\end{equation}
Then, the tangential connection is determined by
$$
\begin{array}{lll}
\nabla^M_{Y_1}Y_1=0, &\quad \nabla^M_{Y_1}Y_2=0, &\quad \nabla^M_{Y_1}Y_3=0, \\[5pt]
\nabla^M_{Y_2}Y_1=0, &\quad \nabla^M_{Y_2}Y_2=-\frac{a}{\beta}Y_1,  &\quad \nabla^M_{Y_2}Y_3=\frac{\beta}{x_3} a Y_2(a) Y_1+\frac{a}{x_3} Y_2, \\[5pt]
\nabla^M_{Y_3}Y_1=0 &\quad \nabla^M_{Y_3}Y_2=0, &\quad \nabla^M_{Y_3}Y_3=\left(-a\frac{x_3}{\beta}\frac{H'_3}{2}-\frac{a}{\beta}H+a\frac{\beta}{x_3}Y_3(a)\right)Y_1+\frac{a}{x_3}Y_3,
\end{array}
$$
{and} the only non-vanishing components of second fundamental form are
$$
h(Y_1,Y_3)=\frac{1}{x_3},\quad
h(Y_2,Y_2)=\frac{1}{\beta},\quad
h(Y_3,Y_3)=\frac{H}{\beta}-\frac{x_3}{\beta}\frac{H'_3}{2}-\frac{\beta}{x_3}Y_3(a),
$$
where we used the symmetry condition for $h$, which yelds $Y_2(a)=0$.

We now  require the immersion of $M$ to be parallel and we get the following condition:
\begin{align}
\label{p4} 0=(\nabla^M _{Y_3} h) (Y_3,Y_3)=\frac{3}{2}\frac{a}{\beta}H'_3+a\frac{x_3}{\beta}\frac{H''_{33}}{2}-a\frac{\beta}{x_3^2}Y_3(a)-\frac{\beta}{x_3}Y_3(Y_3(a)).
\end{align}

As before, we now look for a system of local coordinates $(u_1, u_2, u_3)$ on $M$, such that
\begin{equation}\label{cooMa21}
    \pa {u_1}=Y_1,\qquad 
    \pa {u_2}=Y_2, \qquad
    \pa {u_3}=\alpha Y_2+ \gamma Y_3,
\end{equation}
for some smooth functions $\alpha,\gamma$ on  $M$. Requiring that $[ \pa {u_2},\pa {u_3}]=0$, we get
\begin{align*}
\begin{cases}
   Y_2(\alpha)=-\gamma\frac{a}{x_3},\\[5pt]
   Y_2(\gamma)=0,
\end{cases}
\end{align*}
which is satisfied by 
$$
\alpha=-{u_2}, \quad \gamma=\frac{x_3}{a},
$$
where $a=a(u_3)$.

Denote now by $F: M \rightarrow \bar M, \; (u_1, u_2, u_3)\mapsto(F_1(u_1,u_2,u_3), \ldots, F_4(u_1,u_2,u_3))$ the immersion of the hypersurface in the local coordinates introduced above. By \eqref{frameA2}, we obtain
\begin{equation} \label{eqpaA21}
\begin{array}{ll}
(\pa {u_1} F_1, \pa {u_1} F_2, \pa {u_1} F_3, \pa {u_1} F_4)=& (1,0,0,0),\\[6pt]
(\pa {u_2} F_1, \pa {u_2} F_2, \pa {u_2} F_3, \pa {u_2} F_4)=& -\frac{F_3}{\beta}{(0,0,0,1)},\\[6pt]
(\pa {u_3} F_1, \pa {u_3} F_2, \pa {u_3} F_3, \pa {u_3} F_4)=& (0,\frac{F_3^2}{\beta}\frac{1}{a},-F_3,\frac{F_3}{\beta}u_2).
\end{array}
\end{equation}

Integrating \eqref{eqpaA21} we find
$$
 F_1= u_1+c,\quad F_3=\rho e^{-u_3},\quad F_4= -\frac{\rho}{\beta}u_2 e^{-u_3}+ d,
$$
for some real constants $c, \rho, d$ and we put $a(u_3)=F_3(u_3)\psi(u_3)$ for some smooth positive function $\psi$ on $M$. 

In this case, {as $H=\varepsilon x_3^{2k}$, condition \eqref{p4} rewrites in function of $\psi(u_3)$ {as follows:}
{\begin{equation}\label{eqp}
2\varepsilon k(k+1)\left(\frac{\rho^k}{\beta}\right)^2e^{-2k u_3}=(\psi')^2-2\psi\psi'+\psi\psi''=(\psi(\psi'-\psi))'.
\end{equation}}
A solution to \eqref{eqp} is given in the form {$\psi(u_3)=\lambda e^{-2mu_3}$ where $\lambda$ and $m$ are real constants and we must consider separately the cases $m=k/2$ or $m\neq k/2$.}

{\textbf{Case 2.1: $\boldsymbol{\sin\theta}\mathbf{=0}$ and $\mathbf{m=k/2}$.}}
In this case {$\psi(u_3)=\lambda e^{-k u_3}$ is a solution if and only if $\lambda=\frac{\rho^k}{\beta}$ and $\varepsilon=1$.}

Applying isometries of the ambient space, we find the following parametrizations for $M$:
\begin{align*}
F(u_1, u_2, u_3)=\begin{cases}
\left(u_1, u_3,\rho e^{-u_3}, -u_2 e^{-u_3}+ d\right) \quad &\mathrm{if}\, k=1,\\[4pt]
\left(u_1,\frac{\rho^{1-k}}{k-1}e^{(k-1)u_3},\rho e^{-u_3}, -\frac{\rho}{\beta}u_2 e^{-u_3}+ d\right) \quad &\mathrm{if}\, k\neq1.
\end{cases}
\end{align*}
After a reparametrization and a translation in the $x_4$-direction, which actually is an isometry of ($\bar M$, $g$), we obtain cases (3) and (4) of the theorem.

\medskip
\textbf{Case 2.2: $\boldsymbol{\sin\theta}\mathbf{= 0}$ and $\mathbf{m\neq k/2}$.}
In this case, both sides of equation \eqref{eqp} must vanish. Then, since $k=0$ yields $m=k/2$, which we excluded, we deduce that $k=-1$. This is a special case in the homogeneous class of Siklos spacetimes: the pure radiation solution of Petrov type $N$ with a $G_6$ isometry group (\cite{Defr}). From \eqref{eqp} we deduce that $\psi$ ia a constant and so, $a(u_3)=\rho e^{-u_3}$.

As before, applying isometries of the ambient space, we find the following parametrizations for $M$:
\begin{align*}
F(u_1, u_2, u_3)=\left(u_1,-\frac{\rho}{\beta} e^{-u_3},\rho e^{-u_3}, -\frac{\rho}{\beta}u_2 e^{-u_3}+ d \right).
\end{align*}
Again, after a reparametrization and a translation in the $x_4$-direction we obtain the case (5) of the theorem.}
\end{proof}

\end{document}